\documentclass[12pt]{article}
\usepackage{amssymb,amsbsy,amsmath,amsfonts,amscd,amsthm,colordvi}
\usepackage[english]{babel}
\usepackage{authblk}
\usepackage[T1]{fontenc}
\usepackage{epsfig}
\usepackage{graphicx}
\usepackage{times}
\usepackage{comment}
\usepackage{euscript}
\usepackage{verbatim}
\usepackage{fancyhdr}
\usepackage[usenames,dvipsnames]{color}
\usepackage{subcaption}
\usepackage{tikz}
\usepackage{float}
\usepackage{a4wide}
\usepackage{hyperref}
\def \ve{\varepsilon} % USE

\def \del{\delta}
\def \ved{{\varepsilon,\del}}%
\def \KSD{\K^{\star, \delta}}
\def \lm{\lambda}

\newcommand{\Ab}[1]{(A.#1)}
\def \gr{\nabla}
\def \pt{\partial}

\def \to{\rightarrow}
\def \tow{\rightharpoonup}
\def \eqdef{\stackrel {\rm def} {=}}
\def \be{\begin{equation}}
\def \ee{\end{equation}}
\def \ds{\displaystyle}
\def \R{{\mathbb R}}  
  \def \K{{\mathbb K}}
\def \bs{\boldsymbol}
\newtheorem{theorem}{Theorem}[section]
\newtheorem{proposition}[theorem]{Proposition}
\newtheorem{corollary}[theorem]{Corollary}
\newtheorem{lemma}{Lemma}[section]
\newtheorem{definition}[theorem]{Definition}

\title{Fully homogenized model for immiscible incompressible two-phase flow
through heterogeneous porous media with thin fractures}

\author[1]{Mladen Jurak}
\author[2]{Leonid Pankratov}
\author[1]{Anja Vrba\v{s}ki}
\affil[1]{Faculty of Science, University of Zagreb, Bijeni\v{c}ka 30, 10000 Zagreb, Croatia}
\affil[2]{Department of Mathematics, B. Verkin Institute for Low Temperature Physics and Engineering, 47, av. Lenin, 61103 Kharkov, Ukraine}

\begin{document}

\maketitle

\renewcommand{\thefootnote}{\fnsymbol{footnote}}
\footnotetext{Email addresses: {\tt jurak@math.hr} (Mladen Jurak), {\tt leonid.pankratov@univ-pau.fr} (Leonid Pankratov),
{\tt avrbaski@math.hr} (Anja Vrba\v{s}ki).}
\footnotetext{Partially supported by Ministry of Science, Education and Sports of Republic of Croatia, grant 37-1193086-3226.}
\renewcommand{\thefootnote}{\arabic{footnote}}

{\bf Abstract}.
In this paper we discuss a model describing global behavior of the two phase incompressible flow in fractured
porous media. The fractured media is regarded as a porous medium consisting of two superimposed continua,
a connected fracture system, which is assumed to be thin of order $\ve\del$,
and an $\ve$--periodic system of disjoint matrix blocks.
We derive global behavior of the fractured media by passing to the limit as $\ve\to 0$ and
then as the relative fracture thickness $\del \to 0$, taking into account
that the permeability of the blocks is proportional to $(\ve\del)^2$,
while permeability of the fractures is of order one.
The macroscopic model obtained is then a fully homogenized model, i.e.,
where all the coefficients are calculated in terms of given data and
do not depend on the additional coupling or cell problems. \\
{\bf Keywords}. Homogenization, incompressible two-phase flow, double porosity media, thin fissures. \\
{\bf 2010 Mathematics Subject Classification}. 35B27, 35K65, 35Q35, 74Q15, 76M50, 76S05.

%%%%%%%%%%%%%%%%%%%%%%%%%%%%%%%%%%%%%%%%%%%%%%%%%%%%%%%%%%%%%%%%%%%%%%%%%%%%%%%%%%
\section{Introduction}
\label{sect:intro}
%%%%%%%%%%%%%%%%%%%%%%%%%%%%%%%%%%%%%%%%%%%%%%%%%%%%%%%%%%%%%%%%%%%%%%%%%%%%%%%%%%
\setcounter{equation}{0}
%%%%%%%%%%%%%%%%%%%%%%%%%%%%%%%%%%%%%%%%%%%%%%%%%%%%%%%%%%%%%%%%%%%%%%%%%%%%%%%%%%
A naturally fractured reservoir is a reservoir that contains fracture planes distributed as a
connected network throughout the reservoir.
This type of porous medium is frequently encountered in hydrology and petroleum applications,
for instance the sedimentary rock that composes a hydrocarbon reservoir.
The fluid flow mechanism in such reservoirs has been known to be significantly
different from that of an ordinary, unfractured reservoir.
Specifically, the flow occurs as if the reservoir possessed two porous structures,
one associated to the porous rock, and the other one to the system of fractures.
Accordingly, a naturally fractured reservoir is considered as a porous medium consisting
of two superimposed continua, a discontinuous system of periodically distributed
matrix blocks surrounded by a connected system of thin fissures.
Characteristic features of fractured rocks are that the volume occupied by
the fractures is much smaller than the volume of the pores;
the matrix keeps most of the fluid while the fractures are notably more permeable (see \cite{BZK}).
The fluid exchange between matrix blocks and fractures is a microscale process
whose strong influence on the flow must be embedded in a large scale flow description.
The macroscopic behavior of fluid flow in such porous medium is described by
the so-called double porosity model which was first derived experimentally
as a physical notion and described by several authors in the engineering literature (\cite{BZK}, \cite{WR}).
In standard double porosity model one assumes that
the width of the fractures is of the same order as the block size.
However, the model of \cite{BZK} assumes that
the measure of the fracture set is small with respect to the measure of the pore blocks.
One of the approaches in modeling such problems is therefore
to consider the thickness of the fractures as an additional small parameter.
In this work we consider a double porosity type model
for two-phase incompressible fluid flow in a porous medium with thin fractures.

The first contribution on the derivation of the double porosity model for two-phase
flow in a fractured medium is \cite{ADH90}, where the effective equations of the double porosity
model are established by formal technique of asymptotic expansion for the cases of completely
miscible incompressible flow, and immiscible incompressible two-phase flow.
The double porosity model for immiscible incompressible two-phase flow in a reduced pressure formulation
is rigorously justified by periodic homogenization in \cite{BLM}.
Another result on the two-phase incompressible immiscible flow in fractured porous media is established
in \cite{Yeh06}.
For the displacement of one compressible miscible fluid by another in a naturally fractured reservoir,
the double porosity model was rigorously derived in \cite{Choq1}.
Furthermore, \cite{Arb-exist} and \cite{Yeh02} study the existence of weak
solutions for the two models of the immiscible two-phase flow in fractured porous media.

The method involving only one small parameter $\ve$ in modeling of the thin structures,
now known as method of mesoscopic energy characteristics,
was proposed by E. Khruslov (see, e.g., \cite{MK}).
The method of two small parameters in modeling of periodic thin structures
has been widely used in the mathematical literature
(see, e. g., \cite{BakhPan}, \cite{Cio-Pau}) and applied to various linear elliptic problems.
An important notion of convergence with respect to two small parameters
was introduced by G. Panasenko in \cite{Pan}.
Commutativity of the scheme of the passage to the limit
for two small parameters, $\ve$ and $\del$, was discussed in \cite{Griso, Griso-1} and \cite{BCP}.
However, all these works study problems with the coefficients which are uniformly bounded
and elliptic with respect to the small parameters.
The first result on homogenization of a linear double porosity problem in the case of thin fissures
was obtained in \cite{PR} where the thickness of the fissures as well as
the order of the permeability in the matrix blocks were modeled by one small parameter $\ve$.
That result was recently generalized in \cite{Moj} where several applications were studied.
Independently a singular double porosity model was proposed in \cite{BCP}.
The method of two small parameters $\ve$ and $\del$ for the linear double porosity model
was proposed in \cite{ABGP} and then used in \cite{AGP}
for the homogenization of a degenerate triple porosity model with thin fissures and
in \cite{APP2} for the homogenization of a single phase flow through a porous medium in a thin layer.
Most of these results were presented in the review paper \cite{APR}.
On the other hand, the nonlinear elliptic double porosity type problem for the
fissure set which is not asymptotically small was studied in \cite{PP},
while the nonlinear elliptic double porosity type problem
in domains with thin fissures was studied in \cite{APP1}.
The main feature of the double porosity models with thin fissures, compared to the standard double porosity models,
is that such models do not contain any coupling between the meso- and macro-scale
through the coefficients that depend on additional cell problems.

This paper contains a new homogenization result for the system modeling
immiscible incompressible two-phase flow in a periodic fractured porous medium with thin fractures,
modeled by the two small parameters. The first one, $\ve$, stands for the periodicity of the structure,
and the second one, $\del$, describes the relative thickness of the fissure system.

The paper is organized in the following way. In Section \ref{sect:mesomodel} we set
up the problem which describes the model on the mesoscale (the Darcy scale)
with the coefficients depending on $\ve$ and $\del$.
Then in Section \ref{sect:deltamodel} we present the global double porosity $\del$-model
which has been derived earlier  in \cite{BLM}, \cite{Yeh06} from the mesoscopic problem
and we present a derivation of the imbibition equation.
Section \ref{sect:imb-subsect:linear} is devoted to decoupling the global $\delta$-model
from the system defined on a matrix cell: following \cite{Arb-simpl},
we linearize the imbibition equation and estimate its asymptotic behavior by using the Laplace transform.
Passage to the limit as $\del\to0$ in the global double porosity $\del$-model is performed in
Section \ref{sect:pass_delta}.
Namely, in Subsection \ref{sect:pass_delta-subsect:unif_apr_est}
we obtain the  {\em a priori} estimates for the weak solutions of the problem
with respect to the space and time variables and establish a necessary compactness result.
The main difficulty in derivation of uniform {\em a priori} estimates
is in treatment of the convolution term. In this paper this term is estimated
without an additional step of discretization of the time derivative.
Finally, Subsection \ref{sect:full_homogen} exhibits global fully homogenized model
for immiscible incompressible two-phase flow in double porosity media with thin fractures;
namely, in the limit as $\del\to 0$ we obtain the following integro-differential system
with constant porous medium coefficients and with an additional source term
of the convolution type:
\begin{equation*}
\label{final-hom}
\left\{
\begin{array}[c]{ll}
\ds
\Phi_f \frac{\pt S_f}{\pt t}
-{\rm div}\, \left( \frac{d-1}{d}\, k_f\, \lm_{w,f} (S_f) \gr P_{w,f} \right) =
- \frac{C_m}{d}\,\frac{\pt}{\pt t}\, \big[\,\big(\mathcal{P}(S_f) - \mathcal{P}(S_f^{0}) \big)
\ast \frac{1}{\sqrt{t}} \big], \\[6mm] % \quad {\rm in}\,\, \Omega_{T}
\ds
-\Phi_f \frac{\pt S_f}{\pt t}
-{\rm div}\, \left( \frac{d-1}{d}\, k_f\, \lm_{n,f} (S_f) \gr P_{n,f} \right) =
\frac{C_m}{d}\,\frac{\pt}{\pt t}\, \big[\,\big(\mathcal{P}(S_f) - \mathcal{P}(S_f^{0}) \big)
\ast \frac{1}{\sqrt{t}} \big],
\end{array}
\right.
\end{equation*}
where the involved parameters are defined in terms of the mesoscale  parameters.

Up to our knowledge this is the first rigorous justification of fully homogenized double porosity
model in the framework of the two-phase flow in a reservoir with thin fissures system.

%%%%%%%%%%%%%%%%%%%%%%%%%%%%%%%%%%%%%%%%%%%%%%%%%%%%%%%%%%%%%%%%%%%%%%%%%%%%%%%%%%
\section{Mesoscale model}
\label{sect:mesomodel}
%%%%%%%%%%%%%%%%%%%%%%%%%%%%%%%%%%%%%%%%%%%%%%%%%%%%%%%%%%%%%%%%%%%%%%%%%%%%%%%%%%
\setcounter{equation}{0}
%%%%%%%%%%%%%%%%%%%%%%%%%%%%%%%%%%%%%%%%%%%%%%%%%%%%%%%%%%%%%%%%%%%%%%%%%%%%%%%%%%

We start from a mesoscopic model of the two-phase incompressible flow
defined in a domain with periodic structure, representing a naturally fractured reservoir.
We consider a bounded Lipschitz domain $\Omega \subset \mathbb{R}^d$ ($d = 2, 3$)
which is a union of disjoint cubes congruent to a reference cell $Y = (0, 1)^d$.
The reference cell $Y$ consists of two subdomains,
corresponding to the two types of rock - the matrix, and the fractures.
Moreover, we suppose the relative fracture thickness to be of order $\del$,
where $\del > 0$ is a small parameter. % tending to zero.
In particular, we use the standard Warren-Root model which assumes
that $Y$ consists of an open cube $Y_m^{\del}$ with edge length $1 - \del$,
centered at the center of $Y$, completely surrounded by a connected fracture subdomain $Y_f^{\del}$,
with a piecewise smooth internal boundary $\Gamma^\del$ between the two media in $Y$.
Therefore it is $Y = Y_m^{\del} \cup \Gamma^\del \cup Y_f^{\del}$,
where $|Y_f^{\del}| = O(\del)$ so that $|Y_f^{\del}|\to 0$ as $\del\to0$.
The outward unit normal vector to $Y_m^{\del}$ is denoted by $\bs{\nu}^{\del}$.

The periodic structure of a reservoir is depicted by a small parameter $\ve>0$
representing the characteristic size of the heterogeneities with respect to the size of $\Omega$.
Accordingly, for $\ve>0$ the domain $\Omega$ is assumed to be covered by a pavement of cells $\ve Y$.
For $\del>0$ let ${\bf 1}_m^{\del}(y)$ and ${\bf 1}_f^{\del}(y)$ be
the characteristic functions of $Y_m^{\del}$ and $Y_f^{\del}$, respectively,
extended $Y$--periodically to the whole $\mathbb{R}^d$.
The system of the matrix blocks in $\Omega$, the fractured part of $\Omega$
and the matrix-fracture interface are denoted by $\Omega^\ved_m$, $\Omega^\ved_f$
and $\Gamma^\ved$, respectively. Hence we have
\begin{equation}
\label{omeg12}
\begin{split}
& \Omega^\ved_m \eqdef \left\{x \in \Omega \, :\,
{\bf 1}_m^{\del}\left(\frac{x}{\ve}\right) = 1\right\}, \quad
\Omega^\ved_f \eqdef \left\{x \in \Omega \, :\,
{\bf 1}_f^{\del}\left(\frac{x}{\ve}\right) = 1\right\}
= \Omega \setminus \overline{\Omega^\ved_m}, \\
& \Gamma^\ved \eqdef \pt \Omega^\ved_f \cap \pt \Omega^\ved_m \cap \Omega.
\end{split}
\end{equation}
For simplicity, we assume that $\Omega^\ved_m \cap \pt \Omega = \emptyset$.

%%%%%%%%%%%%%%%%%%%%%%%%%%%%%%%%%%%%%%%%%%%%%%%%%%%%%%%%%%%%%%%%%%%%%%%%%%%%%%%%%%%%%%%%%%%
\begin{figure}[H]

\begin{tikzpicture}[scale=0.8]
\def\rectanglepath{-- ++(1.3cm,0cm)  -- ++(0cm,1.3cm) -- ++(-1.3cm,0cm) -- cycle}

\filldraw[fill=white!30] (-0.1,-0.1) -- (5.9,-0.1)
                                         --  (5.9, 5.9) -- (-0.1, 5.9) -- cycle;

\filldraw[fill=lightgray] (0,  0) \rectanglepath;
\filldraw[fill=lightgray] (1.5,0) \rectanglepath;
\filldraw[fill=lightgray] (3,  0) \rectanglepath;
\filldraw[fill=lightgray] (4.5,0) \rectanglepath;

\filldraw[fill=lightgray] (0,  1.5) \rectanglepath;
\filldraw[fill=lightgray] (1.5,1.5) \rectanglepath;
\filldraw[fill=lightgray] (3,  1.5) \rectanglepath;
\filldraw[fill=lightgray] (4.5,1.5) \rectanglepath;

\filldraw[fill=lightgray] (0,  3.0) \rectanglepath;
\filldraw[fill=lightgray] (1.5,3.0) \rectanglepath;
\filldraw[fill=lightgray] (3,  3.0) \rectanglepath;
\filldraw[fill=lightgray] (4.5,3.0) \rectanglepath;

\filldraw[fill=lightgray] (0,  4.5) \rectanglepath;
\filldraw[fill=lightgray] (1.5,4.5) \rectanglepath;
\filldraw[fill=lightgray] (3,  4.5) \rectanglepath;
\filldraw[fill=lightgray] (4.5,4.5) \rectanglepath;

\draw[<-] (0.5,5.5) -- (0.5,6.5) node[anchor=south] {$\Omega_m^{\varepsilon,\delta}$};
\draw[<-] (0.5,4.4) -- (-0.5,4.4) node[anchor=east] {$\Omega_f^{\varepsilon,\delta}$};

\draw (1.4,6.0)--(1.4,6.3);
\draw (2.9,6.0)--(2.9,6.3);
\draw[<->] (1.4,6.15)--(2.9,6.15);
\draw (2.2,6.3) node {$\varepsilon$};

\draw (6.0,5.9)--(6.3,5.9);
\draw (6.0,4.4)--(6.3,4.4);
\draw[<->] (6.15,4.4)--(6.15,5.9);
\draw (6.4,5.15) node {$\varepsilon$};

\draw (4.5,6.0)--(4.5,6.3);
\draw (4.3,6.0)--(4.3,6.3);
\draw[->] (3.9,6.15)--(4.3,6.15);
\draw[<-] (4.5,6.15)--(4.9,6.15);
\draw (4.4,6.55) node {$\varepsilon\delta$};

\draw (6.0,3.0)--(6.3,3.0);
\draw (6.0,2.8)--(6.3,2.8);
\draw[->] (6.15,2.4)--(6.15,2.8);
\draw[<-] (6.15,3.0)--(6.15,3.4);
\draw (6.6,2.9) node {$\varepsilon\delta$};

\node [below=0.5cm]
at (3.0,0.2)
{$a)$};

%%%%%%%%%%%%%%%%%%%%%%%%%%%%%%%%%%%%%%%%%%%%%%%%%%%%%%%%

\def\rectanglepath{-- ++(1.3cm,0cm)  -- ++(0cm,1.3cm) -- ++(-1.3cm,0cm) -- cycle}

\draw[fill=white!30] (8.5,-0.1) -- (11.5,-0.1) -- (11.5,2.9) -- (8.5,2.9) -- cycle;
\draw[fill=lightgray]   (8.7,0.1) -- (11.3,0.1) -- (11.3,2.7) -- (8.7,2.7) -- cycle;

\draw[->,thin] (8.5,-0.1) -- (12.5,-0.1)  node[anchor=north] {$y_1$};
\draw[->,thin] (8.5,-0.1) -- (8.5,3.9)  node[anchor=east] {$y_d$};
\draw (11.5,-0.1) node[anchor=north] {$1$};
\draw (8.5,2.9) node[anchor=east] {$1$};

\draw[<-] (9.5,2.0) -- (9.5,3.5) node[anchor=south] {$Y_m^{\delta}$};
\draw[<-] (10.5 ,2.8) -- (10.5,3.5) node[anchor=south] {$Y_f^{\delta}$};
\draw[<-] (11.3 ,1.1) -- (12 ,1.1) node[anchor=west] {$\Gamma^{\delta}$};

\draw (11.6,2.9)--(11.9,2.9);
\draw (11.6,2.7)--(11.9,2.7);
\draw[->] (11.75,3.3)--(11.75,2.9);
\draw[<-] (11.75,2.7)--(11.75,2.3);
\draw (12.4,2.8) node {$\delta/2$};

\draw (11.3,3.0)--(11.3,3.3);
\draw (11.5,3.0)--(11.5,3.3);
\draw[->] (11.8,3.15)--(11.5,3.15);
\draw[<-] (11.3,3.15)--(11.0,3.15);
\draw (11.4,3.6) node {$\delta/2$};

\node [below=0.5cm]
at (10.0,0.2)
{$b)$};

\end{tikzpicture}

\caption{a) The domain $\Omega$ with the microstructure. \quad b) The reference cell $Y$.}
\label{fig:ref}

\end{figure}
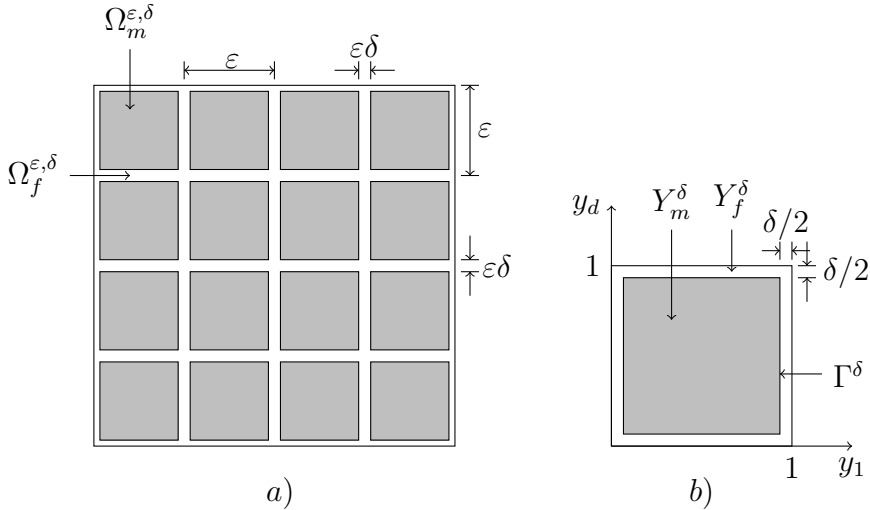
%%%%%%%%%%%%%%%%%%%%%%%%%%%%%%%%%%%%%%%%%%%%%%%%%%%%%%%%%%%%%%%%%%%%%%%%%%%%%%%%%%%%%%%%%%%

The domain boundary $\pt \Omega$ consists of two parts, $\Gamma_{inj}$ and $\Gamma_{imp}$,
such that $\Gamma_{inj} \cap \Gamma_{imp} = \emptyset$,
$\pt \Omega = \Gamma_{inj} \cup \Gamma_{imp}$.
We will use the following notation: $\ell = f, m$ and
$\Omega_T = \Omega \times (0, T)$, $\Omega^\ved_{\ell,T} = \Omega^\ved_\ell \times (0, T)$,
$\Gamma^\ved_{T} = \Gamma^\ved \times (0, T)$, where $T > 0$ is fixed.

In this work we study the incompressible two-phase  flow in porous medium $\Omega$ over the time interval $(0, T)$.
Let $S_{\ell}^\ved \eqdef S_{w, \ell}^\ved$, $S_{n,\ell}^\ved = 1- S_{w,\ell}^\ved$
be the saturations of the wetting and the non-wetting phase in $\Omega^\ved_{\ell,T}$, respectively;
$\lm_{w,\ell} = \lm_{w,\ell}(S_{\ell}^\ved)$,
$\lm_{n,\ell} = \lm_{n,\ell}(S_{\ell}^\ved)$ be the relative mobilities
of the wetting and the non-wetting phase in $\Omega^\ved_{\ell,T}$, respectively;
% and we denote $\lm_{n,\ell} = \lm_{n,\ell}(S_{\ell}^\ved) = \widetilde{\lm}_{n,\ell}(1-S_{\ell}^\ved)$;
let $P_{w,\ell}^\ved$, $P_{n,\ell}^\ved$
be the pressures of the wetting and the non-wetting phase in $\Omega^\ved_{\ell,T}$, respectively.
Finally, let $\Phi^\ved(x)$ and $\K^\ved(x)$ be the porosity and
the absolute permeability tensor of the porous medium $\Omega$ set by
\begin{equation}
\label{fi-ka}
\Phi^\ved(x) \eqdef
\left\{
\begin{array}[c]{ll}
\Phi_f & \; {\rm in}\,\, \Omega^\ved_{f,T} \\[3mm]
\Phi_m & \; {\rm in}\,\, \Omega^\ved_{m,T}  \\
\end{array}
\right.
\; {\rm and } \;\;
\K^\ved(x) \eqdef
\left\{
\begin{array}[c]{ll}
k_f\, \mathbb{I} & \; {\rm in}\,\, \Omega^\ved_{f,T} \\[3mm]
(\ve\delta)^2\,\, k_m\, \mathbb{I} & \; {\rm in}\,\, \Omega^\ved_{m,T}  \\
\end{array}
\right.,
\end{equation}
where $\mathbb{I}$ is the unit tensor.

The mass conservation equations for the individual fluid phases in a subdomain $\Omega^\ved_{\ell,T}$,
$\ell = f,m$, are given by:
\begin{equation}
\label{debut1}
\left\{
\begin{array}[c]{rcl}
\ds
\Phi^\ved(x) \frac{\pt S_{\ell}^\ved}{\pt t} + {\rm div}\, \mathbf{q}^{\ved}_{w,\ell} & = & 0,
%\quad {\rm in}\,\, \Omega^\ved_{\ell,T};
\\[3mm]
\ds
- \Phi^\ved(x) \frac{\pt S_{\ell}^\ved}{\pt t} + {\rm div}\, \mathbf{q}^{\ved}_{n,\ell} & = & 0,
%\quad {\rm in}\,\, \Omega^\ved_{\ell,T},
\\[2mm]
\end{array}
\right.
\end{equation}
with the velocities of the wetting and the non-wetting phases
$\mathbf{q}^{\ved}_{w,\ell}$, $\mathbf{q}^{\ved}_{n,\ell}$ defined by the Darcy-Muskat's law
(see, e.g., \cite{JB-YB}, \cite{CJ}, \cite{Helm}):
\begin{equation}
\label{eq.qw-qn}
\mathbf{q}^{\ved}_{w,\ell} \eqdef -\K^\ved(x) \lm_{w,\ell}(S_{\ell}^\ved) \gr P_{w,\ell}^\ved
%- \vec{g}\right)
, \quad
\mathbf{q}^{\ved}_{n,\ell} \eqdef - \K^\ved(x) \lm_{n,\ell}(S_{\ell}^\ved)\gr P_{n,\ell}^\ved,
% - \vec{g}\right).
\end{equation}
where, for simplicity, the gravity effects are neglected. % terms in \eqref{eq.qw-qn}.

The system \eqref{debut1}-\eqref{eq.qw-qn} is closed by
the capillary pressure law in each of the medium subdomains,
\begin{equation}
\label{eq.pc1}
P_{c,\ell}(S_{\ell}^\ved) = P^\ved_{n,\ell} - P^\ved_{w,\ell},\quad \ell = f, m,
%\quad {\rm with} \,\,
%P^\prime_{c,l}(s) < 0\,\, {\rm for\,\, all}\,\, s \in [0, 1] \,\, {\rm and} \,\, P_{c,l}(1) = 0.
\end{equation}
where $P_{c,\ell}$ is a given capillary pressure-saturation function.

Due to \eqref{fi-ka}, \eqref{eq.qw-qn}, \eqref{eq.pc1},
the system \eqref{debut1} is now written in subdomain $\Omega^\ved_{f,T}$ as
\begin{equation}
\label{debut2-f}
\left\{
\begin{array}[c]{ll}
\ds
\Phi_f \frac{\pt S_f^\ved}{\pt t} -
k_f\,{\rm div}\, \bigg( \lm_{w,f} (S_f^\ved) \gr P_{w,f}^\ved \bigg) = 0, % \quad {\rm in}\,\, \Omega^\ved_{f,T},
\\[6mm]
\ds
- \Phi_f \frac{\pt S_f^\ved}{\pt t} -
k_f\,{\rm div}\, \bigg( \lm_{n,f} (S_f^\ved) \gr P_{n,f}^\ved \bigg) = 0, % \quad {\rm in}\,\, \Omega^\ved_{f,T},
\\[6mm]
P_{c,f}(S_f^\ved) = P_{n,f}^\ved - P_{w,f}^\ved, % \quad {\rm in}\,\, \Omega^\ved_{f,T}.
\\[2mm]
\end{array}
\right.
\end{equation}
and in subdomain $\Omega^\ved_{m,T}$ as
\begin{equation}
\label{debut2-m}
\left\{
\begin{array}[c]{ll}
\ds
\Phi_m \frac{\pt S_m^\ved}{\pt t} -
(\ve\delta)^2\,\, k_m\,{\rm div}\, \bigg( \lm_{w,m} (S_m^\ved) \gr P_{w,m}^\ved \bigg) = 0, % \quad {\rm in}\,\, \Omega_{T},
\\[6mm]
\ds
- \Phi_m \frac{\pt S^\ved_m}{\pt t} -
(\ve\delta)^2\,\, k_m\,{\rm div}\, \bigg( \lm_{n,m} (S_m^\ved) \gr P_{n,m}^\ved \bigg) = 0, % \quad {\rm in}\,\, \Omega_{T},
\\[6mm]
P_{c,m}(S_m^\ved) = P_{n,m}^\ved - P_{w,m}^\ved. % \quad {\rm in}\,\, \Omega_{T}.
\\[2mm]
\end{array}
\right.
\end{equation}

On the matrix-fracture interface $\Gamma^\ved$
the phase fluxes and pressures are required to be continuous:
\begin{equation}
\label{inter-condit}
\left\{
\begin{array}[c]{ll}
\mathbf{q}^{\ved}_{w,f} \cdot \bs{\nu}^{\ved} =
\mathbf{q}^{\ved}_{w,m} \cdot \bs{\nu}^{\ved} \,\,\, {\rm and} \,\,\,
\mathbf{q}^{\ved}_{n,f} \cdot \bs{\nu}^{\ved} =
\mathbf{q}^{\ved}_{n,m} \cdot \bs{\nu}^{\ved}
\quad {\rm on}\,\, \Gamma^\ved_{T}, \\[2mm]
P_{w,f}^\ved = P_{w,m}^\ved \,\, {\rm and} \,\, P_{n,f}^\ved = P_{n,m}^\ved
\quad {\rm on}\,\, \Gamma^\ved_{T}, \\
\end{array}
\right.
\end{equation}
where $\bs{\nu}^{\ved}$ is the unit outward normal vector to $\Gamma^\ved$, directed to $\Omega^\ved_f$.

The boundary conditions for the system \eqref{debut2-f} are given by:
\begin{equation}
\label{bc3}
\left\{
\begin{array}[c]{ll}
P_{w,f}^\ved = P_{w,\Gamma} \quad {\rm and} \quad P_{n,f}^\ved = P_{n,\Gamma} \quad {\rm on} \,\,
\Gamma_{inj} \times (0, T), \\[2mm]
\mathbf{q}^{\ved}_{w,f} \cdot \bs{\nu} =
\mathbf{q}^{\ved}_{n,f} \cdot \bs{\nu} = 0 \quad {\rm on} \,\,
\Gamma_{imp} \times (0, T),\\
\end{array}
\right.
\end{equation}
where $\bs{\nu}$ is the unit outward normal vector to $\pt \Omega$,
and $P_{\alpha, \Gamma}$, $\alpha=w,n$,  are given phase pressures.

The initial conditions read:
\begin{equation}
\label{init1}
S_{f}^\ved(x, 0) = S^{0}_f(x) \,\, {\rm in}\,\, \Omega^\ved_f \quad {\rm and}
\quad
S_{m}^\ved(x, 0) = S^{0}_m(x) \,\, {\rm in} \,\, \Omega^\ved_m.
\end{equation}

Let us now state the following assumptions on data.

\begin{itemize}
\item[\Ab{1}]  % (\underline{\sl Porosity function})
The  porosity coefficients
$0 < \Phi_f,\, \Phi_m < 1$ are constants independent of $\ve$ and $\del$.

\item[\Ab{2}] % (\underline{\sl Absolute permeability tensor})
The  absolute permeability coefficients
$0 < k_f, k_m$ are constants independent of $\ve$ and $\del$.

\item[\Ab{3}] % (\underline{\sl Capillary pressure})
The capillary pressure functions satisfy for $ \ell = f, m$:
$P_{c,\ell} \in C^1((0, 1]; \mathbb{R}^+)$,
$P_{c,\ell}^\prime(s) < 0$ in $(0, 1]$, $P_{c,m}(0^+) = P_{c,f}(0^+) \in (0, \infty],$
$P_{c,\ell}(1) = 0$.
Furthermore, the initial data \eqref{init1} are consistent in the sense that
$P_{c,m} (S_{m}^{0}) = P_{c,f} (S_{f}^{0})$ in $\Omega.$

\item[\Ab{4}] % (\underline{\sl Relative phase mobilities})
The relative phase mobility functions satisfy
$\lm_{w,\ell}, \lm_{n,\ell} \in C([0, 1]; \mathbb{R}^+)$,
$\lm_{w,\ell}(0) = \lm_{n,\ell}(1) = 0$; $0 \leqslant \lm_{w,\ell}, \lm_{n,\ell} \leqslant 1$ in $[0, 1]$;
$\lm_{w,\ell}$ is an increasing function in $[0, 1]$ and $\lm_{n,\ell}$ is a decreasing function in $[0, 1]$.
Moreover, there is a constant $L_0$ such that for all $s \in [0, 1]$,
$\lm_{\ell}(s) \eqdef  \lm_{w,\ell}(s) + \lm_{n,\ell}(s) \geqslant L_0 > 0.$
\end{itemize}

The known theory (see for instance \cite{BH1}, \cite{CP}, \cite{APP})
gives the existence of at least one weak solution to the problem
\eqref{debut2-f}-\eqref{init1} for fixed $\ve>0$, $\del>0$
under the conditions \Ab{1}--\Ab{4}
and some supplementary regularity of saturation functions (see \cite{CP}).

In the following we will use the function
\begin{equation}
\label{mathcalP-def}
\mathcal{P}(s) \eqdef (P_{c,m}^{-1} \circ P_{c,f}) (s)
\end{equation}
that is well defined, monotone increasing and bijective on $[0,1]$ due to \Ab{3}.

%%%%%%%%%%%%%%%%%%%%%%%%%%%%%%%%%%%%%%%%%%%%%%%%%%%%%%%%%%%%%%%%%%%%%%%%%%%%%%%%%%
\section{Global double porosity $\del$-model}
\label{sect:deltamodel}
%%%%%%%%%%%%%%%%%%%%%%%%%%%%%%%%%%%%%%%%%%%%%%%%%%%%%%%%%%%%%%%%%%%%%%%%%%%%%%%%%%
\setcounter{equation}{0}
%%%%%%%%%%%%%%%%%%%%%%%%%%%%%%%%%%%%%%%%%%%%%%%%%%%%%%%%%%%%%%%%%%%%%%%%%%%%%%%%%%

In the case when the typical size of the fractures is
of the same order as the matrix block size, i.e. when $\del = O(1)$,
the homogenization process as $\ve \to 0$
for the mesoscopic model \eqref{debut2-f}-\eqref{init1}
has been studied by formal homogenization techniques in \cite{ADH91}, \cite{AMPP}, \cite{Panf},
and rigorously in \cite{BLM} and \cite{Yeh06}.
More precisely, in \cite{Yeh06} the homogenization procedure
for  problem \eqref{debut2-f}-\eqref{init1}
with a fixed $\del>0$ as $\ve \to 0$ was rigorously justified
by using the notion of the two-scale convergence \cite{All92}.
In this work various, rather strong assumptions were posed on the data
which exclude appearance of one-phase zones and thus degeneracy of the system.
On the other hand, the same type of result for the problem \eqref{debut2-f}-\eqref{init1}
in the global pressure formulation was established in \cite{BLM} %through periodic homogenization,
under an assumption of continuity of the saturations and the global pressure
at the matrix-fracture boundary, but including possible one-phase zones.

We present now the global double porosity $\del$-model which was derived
in \cite{ADH91}, \cite{BLM}, \cite{Yeh06} % Theorem~4.1
by keeping $\del>0$ fixed while passing to the limit as $\ve\to0$ in the
mesoscopic problem \eqref{debut2-f}-\eqref{init1}.
Namely, the global double porosity $\del$-model reads:
\begin{equation}
\label{H-1}
\left\{
\begin{array}[c]{ll}
\ds
\Phi^{\del} \frac{\pt S_f^{\del}}{\pt t} -
{\rm div}\, \bigg( \KSD \lm_{w,f} (S_f^{\del}) \gr P_{w,f}^{\del} \bigg) =
{\EuScript Q}_w^{\del} \quad {\rm in}\,\, \Omega_{T}, \\[6mm]
\ds
- \Phi^{\del} \frac{\pt S_f^{\del}}{\pt t} -
{\rm div}\, \bigg(\KSD \lm_{n,f} (S_f^{\del}) \gr P_{n,f}^{\del} \bigg) =
{\EuScript Q}_n^{\del} \quad {\rm in}\,\, \Omega_{T}, \\[6mm]
P_{c,f}(S_f^{\del}) = P_{n,f}^{\del} - P_{w,f}^{\del} \quad {\rm in}\,\, \Omega_{T}.\\[2mm]
\end{array}
\right.
\end{equation}

The boundary conditions for system \eqref{H-1} are given by:
\begin{equation}
\label{H-6}
\left\{
\begin{array}[c]{ll}
P_{w,f}^{\del} = P_{w,\Gamma} \quad {\rm and} \quad P_{n,f}^{\del} = P_{n,\Gamma} \quad {\rm on} \,\,
\Gamma_{inj} \times (0, T), \\[3mm]
\mathbf{q}^{\,\,\del}_{w,f} \cdot \bs{\nu} =
\mathbf{q}^{\,\,\del}_{n,f} \cdot \bs{\nu} = 0 \quad {\rm on} \,\,
\Gamma_{imp} \times (0, T),\\
\end{array}
\right.
\end{equation}
where
\begin{equation}
\label{H-7}
\mathbf{q}^{\,\,\del}_{w,f} = - \KSD \lm_{w,f}(S_f^{\del}) \gr P_{w,f}^{\del} \quad {\rm and} \quad
\mathbf{q}^{\,\,\del}_{n,f} = - \KSD \lm_{n,f}(S_f^{\del}) \gr P_{n,f}^{\del},
\end{equation}
and the initial condition reads:
\begin{equation}
\label{H-8}
S_f^{\del}(x, 0) = S^{0}_f(x) \,\, {\rm in} \,\, \Omega.
\end{equation}

The effective porosity $\Phi^{\del}$ is given as:
\begin{equation}
\label{H-0}
\Phi^{\del}\eqdef \Phi_f\, \frac{|Y_f^{\del}|}{|Y_m^{\del}|} = \del\, d\, \Phi_f + O(\del^2),
\end{equation}
where $|Y_m^{\del}|$ and $|Y_f^{\del}|$ denote the measure of the set $Y_m^{\del}$ and $Y_f^{\del}$, respectively.
 $\KSD = (\KSD_{ij})$ is the effective permeability tensor given
for $i, j = 1,\ldots,d$ by: % by its elements:
\begin{equation}
\label{H-2}
\KSD_{ij}  \eqdef \frac{k_f}{|Y_m^{\del}|}\,
\int\limits_{Y_f^{\del}}\, \left[\gr_y \xi_i^{\del} + \mathbf{e}_i \right]\,
\left[\gr_y \xi_j^{\del} + \mathbf{e}_j \right]\, dy,
\end{equation}
with $\mathbf{e}_j$ being the $j$--th coordinate vector.
The function $\xi_j^{\del},$ $j=1,\ldots,d,$ is a solution of the cell problem:
\begin{equation}
\label{H-3}
\left\{
\begin{array}[c]{ll}
- \Delta_y \xi_j^{\del} = 0 \quad {\rm in} \,\, Y_f^{\del}, \\[2mm]
(\gr_y \xi_j^{\del} + \mathbf{e}_j)\cdot \bs{\nu}^\del = 0 \quad {\rm on} \,\, \Gamma^{\del},\\[2mm]
y \longmapsto \xi_j^{\del}(y) \quad Y{\rm -periodic}.
\\
\end{array}
\right.
\end{equation}

The matrix-fracture source terms ${\EuScript Q}_w^{\del}$ and ${\EuScript Q}_n^{\del}$ are given by:
\begin{equation}
\label{H-4}
{\EuScript Q}_w^{\del}(x,t) \eqdef - \frac{\Phi_m}{|Y_m^{\del}|} \int\limits_{Y_m^{\del}}
\frac{\pt S_m^{\del}}{\pt t}(x, y, t)\, dy
= - {\EuScript Q}_n^{\del}(x,t),
%\eqdef   \frac{1}{|Y_m^{\del}|} \int\limits_{Y_m^{\del}} \Phi_m\,
%\frac{\pt S_m^{\del}}{\pt t}(x, y, t)\, dy,
\end{equation}
where the function $S_m^{\del}$ is the matrix block saturation defined below.

To each point $x \in \Omega$ there is an associated matrix block congruent to $Y_m^{\del}$.
For any $x \in \Omega$ the flow equations in a matrix block $Y_m^{\del} \times (0, T)$ are:
\begin{equation}
\label{H-5}
\left\{
\begin{array}[c]{ll}
\ds
\Phi_m \frac{\pt S_m^{\del}}{\pt t} -
\del^2 k_m\, {\rm div}_y\, \bigg( \lm_{w,m} (S_m^{\del}) \gr_y P_{w,m}^{\del}\bigg) = 0,
% \quad {\rm in}\,\, \Omega_{T} \times Y_m^{\del};
\\[5mm]
\ds
- \Phi_m \frac{\pt S_m^{\del}}{\pt t} -
\del^2 k_m\, {\rm div}_y\, \bigg( \lm_{n,m} (S_m^{\del}) \gr_y P_{n,m}^{\del} \bigg) = 0,
% \quad {\rm in}\,\, \Omega_{T} \times Y_m^{\del};
\\[5mm]
P_{c,m}(S_m^{\del}) = P_{n,m}^{\del} - P_{w,m}^{\del}.
% \quad {\rm in}\,\, \Omega_{T} \times Y_m^{\del}.
\\[2mm]
\end{array}
\right.
\end{equation}
On the interface $\Gamma^\del$ in the cell $Y$ we have the continuity conditions for any $x \in \Omega$:
\begin{equation}
\label{H-9}
P_{w,m}^{\del}(x, y, t) =  P_{w,f}^{\del}(x, t) \quad {\rm and} \quad
P_{n,m}^{\del}(x, y, t) =  P_{n,f}^{\del}(x, t) \quad {\rm on} \,\, \Gamma^\del \times (0, T).
\end{equation}
Finally, the initial condition is
\begin{equation}
\label{H-9-a}
S_m^{\del}(x, y, 0) = S^{0}_m(x) \,\, {\rm in} \,\,  \Omega \times Y_m^{\del}.
\end{equation}

The existence of weak solutions of the global $\del$-problem {\rm \eqref{H-1}-\eqref{H-9-a}}
is a consequence of the homogenization result in \cite{BLM}, \cite{Yeh06}
and it has also been studied in \cite{Yeh02}.

It can be seen, as in \cite{Cio-Pau}, that there exist positive constants $\hat{k}_m^1$, $\hat{k}_m^2$
such that the effective permeability tensor $\KSD$ satisfies for any $\bs{\xi} \in \R^d:$
\begin{equation}
\label{H-2-unif-5-0}
\hat{k}_m^1\, |\bs{\xi}|^2 \leq \frac{1}{\del}\, \KSD \bs{\xi}\cdot\bs{\xi} \leq \hat{k}_m^2\, |\bs{\xi}|^2.
\end{equation}
Following \cite{Cio-Pau}, Chapter 2, the asymptotic behavior of the homogenized permeability tensor $\KSD$
with respect to $\del$ is given by
\begin{equation}
\label{H-2-unif-5}
\frac{\KSD_{ij}}{|Y^\del_f|} = \K^{\star}_{ij} + \bar{\K}^{\del}_{ij},
\end{equation}
where $\bar{\K}^{\del}_{ij} \to 0$ and $|Y^\del_f| = d\,\del +O(\del^2)$.
Moreover (see \cite{Cio-Pau}), the tensor $\K^{\star}=(\K^{\star}_{ij})$ can be calculated as
\begin{equation}
\label{H-2-unif-5-1}
\K^{\star} = k^*\, \mathbb{I} \quad {\rm with}\,\, k^* = \frac{d-1}{d}\, k_f,\quad d=2,3.
\end{equation}

The problem \eqref{H-5}-\eqref{H-9-a} can be simplified
due to the constant in $y$ boundary conditions by eliminating the matrix phase pressures as follows.
Let us first introduce for $\ell = f, m$ the functions
\begin{equation}
\label{upsi-1}
\beta_\ell(s) \eqdef\int\limits_0^s \alpha_\ell(\xi)\, d\xi,
\quad \textrm{ where } \, \,
\alpha_\ell(s) \eqdef \frac{\lm_{w,\ell}(s)\, \lm_{n,\ell}(s)} {\lm_\ell(s)} | P^\prime_{c,\ell}(s) |.
\end{equation}

\begin{lemma}
\label{lemma-der-imb-eqn}
Let $S_m^{\del}(x,y,t)$ be the solution of the cell problem \eqref{H-5}-\eqref{H-9-a}.
It holds:
\begin{equation}
\label{imb-eqn}
\left\{
\begin{array}[c]{ll}
\ds
\Phi_m \frac{\pt S_m^{\del}}{\pt t}\, - \del^2\, k_m\, \Delta_y \beta_m(S_m^{\del}) = 0
\quad {\rm in}\,\, \Omega_T \times Y_m^{\del}, \\[4mm]
S_m^{\del}(x, y, t) = \mathcal{P}(S_f^{\del}(x,t)) \quad {\rm on}\,\, \Omega_T \times \Gamma^{\del}, \\[4mm]
S_m^{\del}(x, y, 0) = S_m^0(x) \quad {\rm in}\,\, \Omega \times Y_m^{\del}.\\[2mm]
\end{array}
\right.
\end{equation}
\end{lemma}
Equation $\eqref{imb-eqn}_1$ is known as the {\em imbibition equation}.

\begin{proof}
Let us first introduce the global pressure $\mathsf{P}_m^{\del}$ in the matrix block (see \cite{AKM,CJ}) by
\begin{equation}
\label{gp1}
P_{w,m}^{\del} = \mathsf{P}_m^{\del} - % + {\mathsf G}_{w,m}(1) -
\int\limits_{S_m^{\del}}^1 \frac{\lm_{n,m}(\xi)} {\lm_{m}(\xi)} \, P_{c,m}^\prime(\xi)\, d\xi,
\quad % {\rm and} \quad
P_{n,m}^{\del} = \mathsf{P}_m^{\del} + % + {\mathsf G}_{n,m}(1) +
\int\limits_{S_m^{\del}}^1 \frac{\lm_{w,m}(\xi)} {\lm_{m}(\xi)} \, P_{c,m}^\prime(\xi)\, d\xi,
\end{equation}
where the total mobility function $\lm_m$ is defined  in \Ab{4}.
From the boundary conditions \eqref{H-9} at the interface $\Gamma^\del$
we immediately get $\eqref{imb-eqn}_2$.
Since the function $S_m^{\del}$ does not depend on $y$ on $\Omega_T \times \Gamma^\del$,
it follows that the global pressure $\mathsf{P}_m^{\del}$ does not depend on $y$ on $\Omega_T \times \Gamma^\del$.
Therefore, we can write
\begin{equation}
\label{H-9-then-2}
\mathsf{P}_m^{\del}(x, y, t) = P_{m,\Gamma}^{\del}(x, t) \quad {\rm on} \,\, \Omega_T \times \Gamma^\del.
\end{equation}
By summing the two equations in  problem \eqref{H-5} and by applying the
definition of $\mathsf{P}_m^{\del}$ we get (\cite{AKM,CJ})
\begin{equation}
\label{H-18}
 -\del^2\, k_m\, {\rm div}\, \bigg( \lm_{m} (S_m^{\del})\gr \mathsf{P}_{m}^{\del} \bigg) =
0 \quad {\rm in}\,\, \Omega_{T} \times Y_m^\del,
\end{equation}
and by multiplying the equation \eqref{H-18} by $\mathsf{P}_m^{\del} - P_{m,\Gamma}^{\del}$
and integrating over $\Omega_T\times Y_m^{\del}$,
using \eqref{H-9-then-2} and \Ab{4} we obtain:
$$
0 = \del^2\, k_m\, \int\limits_{\Omega_T\times Y_m^{\del}}
\lm_m (S_m^{\del})|\gr_y \mathsf{P}_m^{\del}|^2 \, dx\,dy\,dt \geqslant
\del^2\, k_m\,\, L_0\, \int\limits_{\Omega_T\times Y_m^{\del}} |\gr_y \mathsf{P}_m^{\del}|^2 \, dx\,dy\,dt,
$$
which gives
\begin{equation}
\label{H-9-then-5}
\gr_y \mathsf{P}_m^{\del} = 0 \quad {\rm a.e.\,\, in} \,\, \Omega_T\times Y_m^{\del}.
\end{equation}

This result allows us to reduce the two equations in the problem \eqref{H-5} to only one, as announced.
Namely, by taking into account \eqref{H-9-then-5} and the identity
$$
\lm_{w,m}(S_m^{\del}) \gr_y P_{w,m}^{\del} =
\lm_{w,m}(S_m^{\del}) \gr_y \mathsf{P}_m^{\del} - \gr_y \beta_m(S_m^{\del}),
$$
from $\eqref{H-5}_1$ we establish $\eqref{imb-eqn}_1$.
This completes the proof of Lemma~\ref{lemma-der-imb-eqn}.

\end{proof}

Let us point out that the matrix-fracture source terms ${\EuScript Q}_w^{\del}, {\EuScript Q}_n^{\del}$
of the system \eqref{H-1}, given in an implicit form by \eqref{H-4},
involve the function $S_m^{\del}$ which is a solution
of the local boundary value problem
\eqref{imb-eqn}, which is coupled with the global problem \eqref{H-1}-\eqref{H-8}
through its boundary condition.
This feature of the system \eqref{H-1}-\eqref{H-4}, \eqref{imb-eqn}
is captured by the concept introduced in \cite{AMPP}:
the homogenized system of equations is said to be {\bf fully homogenized}
if it does not involve the unknown functions which are defined
as the solutions of the coupled local problems.
The global $\del$-problem \eqref{H-1}-\eqref{H-4}, \eqref{imb-eqn}
is not fully homogenized in the said sense.
The purpose of the 	succeeding sections is to express the source terms
${\EuScript Q}_w^{\del}, {\EuScript Q}_n^{\del}$ in an explicit form
by decoupling the global system \eqref{H-1}-\eqref{H-4}
from the local problem \eqref{imb-eqn}.
This will be done by passing to the limit as $\del\to0$ in the system \eqref{H-1}-\eqref{H-4}, \eqref{imb-eqn}
and thereby establishing the fully homogenized model.
Following the idea of \cite{Arb-simpl} we will first linearize the imbibition equation \eqref{imb-eqn}
and perform the asymptotic analysis of the linearized imbibition equation.

%%%%%%%%%%%%%%%%%%%%%%%%%%%%%%%%%%%%%%%%%%%%%%%%%%%%%%%%%%%%%%%%%
\section{Linearized imbibition equation}
\label{sect:imb-subsect:linear}
%%%%%%%%%%%%%%%%%%%%%%%%%%%%%%%%%%%%%%%%%%%%%%%%%%%%%%%%%%%%%%%%%
\setcounter{equation}{0}
%%%%%%%%%%%%%%%%%%%%%%%%%%%%%%%%%%%%%%%%%%%%%%%%%%%%%%%%%%%%%%%%%

Our next step is to simplify the matrix cell problem \eqref{imb-eqn}
by introducing a linearized version of that problem.
As suggested in \cite{Arb-simpl}, we consider a function $\psi_m(x)$
such that
\begin{equation}
\label{psi_m-def}
\psi_m \thickapprox \alpha_m(S_m^\del).
\end{equation}
Moreover, we assume that there are constants $\psi_{m}^{min}, \psi_{m}^{max}$
such that for any $x \in \Omega$ it holds
\begin{equation}
\label{psi_m-bound}
0 < \psi_{m}^{min} \leq \psi_m(x) \leq \psi_{m}^{max}.
\end{equation}

Thus we replace the imbibition equation \eqref{imb-eqn}
by its linearized version
\begin{equation}
\label{S_m_L_delta}
\left\{
\begin{array}[c]{ll}
\ds
\Phi_m \frac{\pt S_{m}^{\del}}{\pt t}\, - \del^2 k_m\, \psi_m(x) \Delta_y S_{m}^{\del} = 0
\quad {\rm in}\,\, \Omega_T \times Y_m^{\del}, \\[4mm]
S_{m}^{\del}(x, y, t) = \mathcal{P}(S_f^{\del}(x,t)) \quad {\rm on}\,\, \Omega_T \times \Gamma^{\del}, \\[4mm]
S_{m}^{\del}(x, y, 0) = S_m^0(x) \quad {\rm in}\,\, \Omega \times Y_m^{\del}.\\[2mm]
\end{array}
\right.
\end{equation}

The particular choice of function $\psi_m$ was proposed and validated in \cite{Arb-simpl}.
The numerical simulations were performed for exact and linearized models
and the computational results show that the linearized model is computationally less complex
while essentially without  significant loss in accuracy
compared to the exact model.

An existence result for the model \eqref{H-1}-\eqref{H-4}, \eqref{S_m_L_delta} is proved in \cite{Arb-exist}.

In order to analyze the behavior of $S_m^{\del}$ as $\del\to0$ we replace
parabolic problem \eqref{S_m_L_delta} by an elliptic problem by use of the
Laplace transform $\mathcal{L}$.
Let $S_{m}^{\del}(x,y,t)$ be the solution of the linearized problem \eqref{S_m_L_delta}.
We denote for $\lm >0$:
$$s_{m}^{\del} \eqdef \mathcal{L}(S_{m}^{\del}).$$

By using the basic properties of the Laplace transformation, it follows easily that
the function $s_{m}^{\del}(x,y,\lm)$ satisfies the following problem:
\begin{equation}
\label{S_m_L_lm_delta}
\left\{
\begin{array}[c]{ll}
\ds
\lm\Phi_m \, s_{m}^{\del}(x,y,\lm)
- \del^2 k_m \psi_m(x) \Delta_y s_{m}^{\del}(x,y,\lm)
= \Phi_m S_m^0(x) \quad {\rm in}\,\, \Omega \times Y_m^{\del}, \\[4mm]
s_{m}^{\del}(x,y,\lm) = \mathcal{L}\big(\,\mathcal{P}(S_f^{\del})\,\big)(x,\lm) \quad {\rm on}\,\, \Omega \times \Gamma^{\del}.
\end{array}
\right.
\end{equation}
Introducing the associated auxiliary problem with constant boundary data:
\begin{equation}
\label{U_lm_1_delta}
\left\{
\begin{array}[c]{ll}
\ds
\lm\Phi_m \, \mathsf{u}^{\del}(x,y,\lm)
- \del^2 k_m \psi_m(x) \Delta_y \mathsf{u}^{\del}(x,y,\lm)
= 0 \quad {\rm in}\,\, \Omega \times Y_m^{\del}, \\[4mm]
\mathsf{u}^{\del}(x,y,\lm)  = 1 \quad {\rm on}\,\, \Omega \times \Gamma^{\del},
\end{array}
\right.
\end{equation}
it is easy to see that the solution $s_{m}^{\del}$  of \eqref{S_m_L_lm_delta} is  given by
\begin{equation}
\label{U_lm_delta-sol}
s_{m}^{\del}(x,y,\lm) = \frac{1}{\lm}\, S_m^0(x)
+\mathsf{u}^{\del}(x,y,\lm)\,\, \mathcal{L}\big(\,\mathcal{P}(S_f^{\del}(x,t)) - S_m^0(x)\,\big).
\end{equation}

From \eqref{H-4}, using \eqref{U_lm_delta-sol} we obtain
\begin{equation}
\label{Qw-2}
{\EuScript Q}_{w}^{\del}(x,t) = - \frac{\Phi_m}{|Y_m^{\del}|}\,
\mathcal{L}^{-1}\bigg(\, \lm\, \mathcal{L}\big(\,\mathcal{P}(S_f^{\del}(x,t)) - S_m^0(x)\,\big)\, \int_{Y_m^{\del}} \mathsf{u}^{\del}(x,y,\lm) dy \bigg).
\end{equation}

In order to estimate the asymptotic behavior of ${\EuScript Q}_{w}^{\del}$ as $\del$ tends to $0$,
we need to estimate asymptotically in $\del$ the integral term
$\ds\int_{Y_m^{\del}} \mathsf{u}^{\del}(x,y,\lm) dy$ in \eqref{Qw-2}.
Slightly modifying the proof of Lemma~7.2 from \cite{PR}, we have:

\begin{lemma}
\label{l-asympt}
For any $x\in\Omega$, let $\mathsf{u}^{\del}(x,y,\lm)$
be the solution of the problem \eqref{U_lm_1_delta} with parameter $x$.
Then it holds as $\del\to0$, uniformly in $x$,
\begin{equation}
\label{U_lm_1_delta_asympt}
\ds
\int_{Y_m^{\del}} \mathsf{u}^{\del}(x,y,\lm)dy =
\frac{6 \sqrt{k_m \psi_m(x)}}  {\sqrt{\Phi_m} \sqrt{\lm}}\, \del\, (1+o(1)).
\end{equation}
\end{lemma}

Finally, from \eqref{Qw-2} and \eqref{U_lm_1_delta_asympt},
by applying the basic properties of the Laplace transformation,
we obtain the following result.

\begin{corollary}
\label{cor-asympt}
The simplified matrix-fracture source terms ${\EuScript Q}_{w}^{\del}, {\EuScript Q}_{n}^{\del}$
satisfy
\begin{equation}
\label{source-conv}
{\EuScript Q}_{w}^{\del}(x,t) =
- \frac{\pt}{\pt t}\, \big[\,\big(\mathcal{P}(S_f^{\del}) - \mathcal{P}(S_f^{0}) \big)
\ast \omega^\del \big] (x,t)
= - {\EuScript Q}_{n}^{\del}(x,t),
\end{equation}
where we denote
\begin{equation}
\label{w_Ddelta-def}
%\hat{\omega}(t) = \frac{1}{\sqrt{t}}; \quad
\omega^\del(x,t) \eqdef D^{\del}(x)\, \frac{1}{\sqrt{t}}, \quad
D^{\del}(x) \eqdef \del\, \left( \frac{C_m(x)}{|Y_m^{\del}|} + o(1) \right), \quad
C_m(x) \eqdef \frac{6 \sqrt{\Phi_m k_m \psi_m(x)}}{\sqrt{\pi}},
\end{equation}
and $\ast$ denotes convolution with respect to time:
$\ds (f \ast g)(t) \eqdef \int_0^t\, f(\tau) g(t-\tau) d\tau$.
\end{corollary}

Note that for all $x \in \Omega$ and sufficiently small $\del$ it holds
\begin{equation}
\label{w_Ddelta-max}
D^{\del}(x) \leq 2\,\del\, C_m^{max},\quad  \ds C_m^{max} = \frac{6 \sqrt{\Phi_m k_m \psi_m^{max}}}{\sqrt{\pi}}.
\end{equation}

%%%%%%%%%%%%%%%%%%%%%%%%%%%%%%%%%%%%%%%%%%%%%%%%%%%%%%%%%%%%%%%%%%%%%%%%%%%%%%%%%%
\section{Passage to the limit as $\del \to 0$}
\label{sect:pass_delta}
%%%%%%%%%%%%%%%%%%%%%%%%%%%%%%%%%%%%%%%%%%%%%%%%%%%%%%%%%%%%%%%%%%%%%%%%%%%%%%%%%%
\setcounter{equation}{0}
%%%%%%%%%%%%%%%%%%%%%%%%%%%%%%%%%%%%%%%%%%%%%%%%%%%%%%%%%%%%%%%%%%%%%%%%%%%%%%%%%%

In order to derive the fully homogenized model we need to pass to the limit
as $\del\to0$ in the problem \eqref{H-1} with corresponding boundary and initial conditions.
We start by transforming the system \eqref{H-1}
by employing new variables: the global pressure $\mathsf{P}_f^{\del}$
and a "complementary pressure" $\theta_f^{\del}$.
First the global pressure $\mathsf{P}_f^{\del}$ in the fractures is inducted
analogously to \eqref{gp1} by
\begin{equation}
\label{gp1-f}
P_{w,f}^{\del} = \mathsf{P}_f^{\del} % + {\mathsf G}_{w,f}(1) -
- \int\limits_{S_f^{\del}}^1 \frac{\lm_{n,f}(\xi)} {\lm_f(\xi)} \, P_{c,f}^\prime(\xi)\, d\xi,
\quad % {\rm and} \quad
P_{n,f}^{\del} = \mathsf{P}_f^{\del} % + {\mathsf G}_{n,f}(1) +
+ \int\limits_{S_f^{\del}}^1 \frac{\lm_{w,f}(\xi)} {\lm_f(\xi)} \, P_{c,f}^\prime(\xi)\, d\xi.
\end{equation}
A "complementary pressure" $\theta_f^{\del}$ is defined (see \cite{Arb-exist}) by
\begin{equation}
\label{comp-pres-def}
\ds
\theta_f^{\del} \eqdef \beta_f (S_f^{\del}),
%- \int\limits_{0}^{S_f^{\del}} \left( \frac{\lm_{w,f} \lm_{n,f}} {\lm_{f}^2} \, %P^\prime_{c,f}\right) (\xi)\, d\xi
\end{equation}
where $\beta_f$ is defined in \eqref{upsi-1}.
We denote
$\ds \theta_f^{\star} = \beta_f (1)$
and the inverse function
\begin{equation}
\label{mathcalS-def}
\ds
S_f^{\del} = \mathcal{B}_f(\theta_f^{\del}) \eqdef \beta_f^{-1}(\theta_f^{\del})
\quad {\rm for}\, \ 0 \leq \theta_f^{\del} \leq \theta_f^{\star}.
\end{equation}
Note that $\mathcal{B}_f: [0, \theta_f^*] \to [0,1]$ is a continuous and monotone increasing function.

Finally, the system  \eqref{H-1}
in terms of the global pressure and the complementary pressure reads:
\begin{equation}
\label{H-delta-lin-GP}
\left\{
\begin{array}[c]{ll}
\ds
 {\rm div}\, \bigg( \lm_{f} (S_f^{\del})\mathbb{K}^{\star,\del} \gr \mathsf{P}_f^{\del} \bigg) =
0 \quad {\rm in}\,\, \Omega_{T}, \\[6mm]
\ds
\Phi^{\del} \frac{\pt S_f^{\del}}{\pt t}
- {\rm div}\, \bigg( \mathbb{K}^{\star, \del} \left( \lm_{f} (S_f^{\del}) \gr \theta_{f}^{\del}
+ \lm_{w,f} (S_f^{\del}) \gr \mathsf{P}_f^{\del} \right) \bigg) =
{\EuScript Q}_{w}^{\del} \quad {\rm in}\,\, \Omega_{T}, \\[6mm]
S_f^{\del} = \mathcal{B}_f(\theta_f^{\del}) \quad {\rm in}\,\, \Omega_{T}.\\[2mm]
\end{array}
\right.
\end{equation}
The boundary conditions for system \eqref{H-delta-lin-GP} are given by:
\begin{equation}
\label{H-14}
\left\{
\begin{array}[c]{ll}
\mathsf{P}_f^{\del} = P_{\Gamma} \quad {\rm and} \quad \theta_{f}^{\del} = \theta_{\Gamma} \enskip {\rm on} \,\,
\Gamma_{inj} \times (0, T), \\[3mm]
 \KSD \lm_{f}(S_f^{\del}) \gr \mathsf{P}_f^{\del} \cdot \bs{\nu} = 0  \ {\rm on} \,\,
\Gamma_{imp} \times (0, T), \\[3mm]
 \KSD \left( \lm_{f} (S_f^{\del}) \gr \theta_{f}^{\del}
+ \lm_{w,f} (S_f^{\del}) \gr \mathsf{P}_f^{\del} \right) \cdot \bs{\nu} = 0 \ {\rm on} \,\,
\Gamma_{imp} \times (0, T).
\end{array}
\right.
\end{equation}
The initial condition reads:
\begin{equation}
\label{H-15}
\theta_f^{\del}(x, 0) = \theta^{0}_f(x) \,\, {\rm in} \,\, \Omega.
\end{equation}
The boundary and initial data $P_{\Gamma}$, $\theta_{\Gamma}$ and $\theta^{0}_f$ in \eqref{H-14} and \eqref{H-15}
are given by the corresponding transformations of the functions $P_{w,\Gamma}$, $P_{n,\Gamma}$,
$S_{f,\Gamma} \eqdef P^{-1}_{c,f}(P_{n,\Gamma} - P_{w,\Gamma})$ % (which is defined in \eqref{bou-fun-1})
and $S^{0}_f$.

Now we state the rest of the assumptions on the data which will assure the existence for weak solutions
of the problem \eqref{H-delta-lin-GP}-\eqref{H-15}.

\begin{itemize}
\item[\Ab{5}] % (\underline{\sl Initial data})
The boundary and initial data satisfy: $P_{\Gamma} \in L^2(0,T;H^1(\Omega))$,
$\theta_{\Gamma} \in L^2(0,T;H^1(\Omega))$, $\pt_t \theta_{\Gamma} \in L^1(\Omega_T)$,
$\theta_f^0 \in L^2(\Omega)$, $0 \leq \theta_f^0, \theta_{\Gamma} \leq \theta_f^* \textrm { a.e. in } \Omega$.
\end{itemize}

A weak solution of this problem is defined as follows.
Let
\begin{equation}
\label{V-def}
V = \{u\in H^1(\Omega),u|_{\Gamma_{inj}}=0\}.
\end{equation}

\begin{definition}
\label{def-weak-delta-lin}
A weak solution to the system \eqref{H-delta-lin-GP}-\eqref{H-15}
is a pair $(\mathsf{P}_f^{\del}, \theta_{f}^{\del})$ such that
\begin{equation*}
\label{weak-sol-spaces}
\left.
\begin{array}[c]{cc}
\mathsf{P}_f^{\del} - P_{\Gamma} \in L^2(0,T;V), \enskip \theta_{f}^{\del} - \theta_{\Gamma} \in L^2(0,T;V), \enskip
0 \leq \theta_f^{\del} \leq \theta_f^* \quad \textrm{a.e. in } \Omega_T, \enskip
S_f^{\del} = \mathcal{B}_f(\theta_f^{\del}),\\ [2mm]
\ds
\pt_t\, \big( \Phi^{\del} S_f^{\del} + \,\big(\mathcal{P}(S_f^{\del}) - \mathcal{P}(S_f^{0}) \big) \ast \omega^\del \big) \in L^2(0,T;V'),\\
\end{array}
\right.
\end{equation*}
{\it for any $\zeta, \varphi \in L^2(0,T;V)$}
\begin{equation}
\label{H-16-weak}
\int\limits_{\Omega_T} \lm_{f} (S_f^{\del})\mathbb{K}^{\star, \del} \gr \mathsf{P}_f^{\del}
\cdot \gr\zeta\, dx\, dt = 0,
\end{equation}
\begin{equation}
\label{H-17-weak}
\begin{split}
\int\limits_{0}^T \langle \pt_t \big(\Phi^{\del} S_f^{\del}
+ \,\big(\mathcal{P}(S_f^{\del}) & - \mathcal{P}(S_f^{0}) \big) \ast \omega^\del  \big),
\varphi \rangle  dt \\
+ \int\limits_{\Omega_T} & \mathbb{K}^{\star, \del} \left( \lm_{f} (S_f^{\del}) \gr \theta_{f}^{\del}
+ \lm_{w,f} (S_f^{\del}) \gr \mathsf{P}_f^{\del} \right) \cdot\gr\varphi\, dx\, dt = 0.
\end{split}
\end{equation}
{\it Furthermore, the initial condition is satisfied in the following sense:\\
\indent for any $ \varphi \in L^2(0,T;V) \cap W^{1,1}(0,T;L^1(\Omega))$ such that $\varphi(\cdot, T) = 0$ in $\Omega$,}
\begin{equation}
\label{H-weak-init}
\begin{split}
\int\limits_{0}^T \langle \pt_t \big(\Phi^{\del} S_f^{\del}
+ \,\big(\mathcal{P}(S_f^{\del}) & - \mathcal{P}(S_f^{0}) \big) \ast \omega^\del  \big),
\varphi \rangle  dt \\
+ \int\limits_{\Omega_T} &
\left( \Phi^{\del} ( S_f^{\del} - S_f^0 ) + \,\big(\mathcal{P}(S_f^{\del}) - \mathcal{P}(S_f^{0}) \big) \ast \omega^\del \right) \pt_t \varphi\, dx\, dt = 0.
\end{split}
\end{equation}
\end{definition}

The existence of a weak solution from Definition~\ref{def-weak-delta-lin}
under conditions \Ab{1}--\Ab{5} follows from the result of \cite{Arb-exist}, Theorem~1.
Our goal is to pass to the limit as $\del \to 0$ in the system
\eqref{H-delta-lin-GP}-\eqref{H-15}.

%%%%%%%%%%%%%%%%%%%%%%%%%%%%%%%%%%%%%%%%%%%%%%%%%%%%%%%%%%%%%%%%%
\subsection{Uniform a priori estimates}
\label{sect:pass_delta-subsect:unif_apr_est}
%%%%%%%%%%%%%%%%%%%%%%%%%%%%%%%%%%%%%%%%%%%%%%%%%%%%%%%%%%%%%%%%%

First we establish the following uniform estimates.

\begin{proposition}
\label{Prop-uniform-est}
Let $\del>0$.
Let $(\mathsf{P}_f^{\del}, \theta_{f}^{\del})$
be a weak solution of the problem \eqref{H-delta-lin-GP}-\eqref{H-15}.
The following estimates, uniform with respect to $\del$, hold:
\begin{equation}
\label{unif-est-GP-1}
\left.
\begin{array}[c]{cc}
\Vert \mathsf{P}_f^{\del} \Vert_{L^2(0,T;H^1(\Omega))} + \Vert \theta_f^{\del} \Vert_{L^2(0,T;H^1(\Omega))} \leq C,\\
\end{array}
\right.
\end{equation}
\begin{equation}
\label{unif-est-GP-2}
\left.
\begin{array}[c]{cc}
\ds
\Vert \frac{1}{\del}\, \pt_t\, \big( \Phi^{\del} S_f^{\del} + \,\big(\mathcal{P}(S_f^{\del}) - \mathcal{P}(S_f^{0}) \big) \ast \omega^\del \big) \Vert_{L^2(0,T;V')} \leq C.
\end{array}
\right.
\end{equation}
\end{proposition}
\begin{proof}
We first insert $\zeta = \mathsf{P}_f^{\del} - P_{\Gamma}$ into the equation \eqref{H-16-weak}.
This yields
\begin{equation}
\label{H-discrete-3-1}
\int\limits_{\Omega_T} \lm_{f} (S_f^{\del})\mathbb{K}^{\star, \del} |\gr \mathsf{P}_f^{\del}|^2\, dx\, dt =
\int\limits_{\Omega_T} \lm_{f} (S_f^{\del})\mathbb{K}^{\star, \del} \gr \mathsf{P}_f^{\del} \cdot \gr P_{\Gamma}\, dx\, dt.
\end{equation}
Taking into account the representation \eqref{H-2-unif-5} of the tensor $\KSD$, we get
\begin{equation}
\label{H-discrete-4-1}
\begin{split}
\int\limits_{\Omega_T} \lm_{f} (S_f^{\del})\K^{\star} |\gr \mathsf{P}_f^{\del}|^2\, dx\, dt & =
\frac{1}{|Y_f^\del|}\int\limits_{\Omega_T} \lm_{f} (S_f^{\del})\KSD \gr \mathsf{P}_f^{\del} \cdot \gr P_{\Gamma}\, dx\, dt\\
&- \int\limits_{\Omega_T} \lm_{f} (S_f^{\del}) \bar{\K}^{\del} |\gr \mathsf{P}_f^{\del}|^2\, dx\, dt,
\end{split}
\end{equation}
and by applying \eqref{H-2-unif-5-1} and \Ab{4},
we finally obtain
\begin{equation}
\label{H-discrete-5-1}
\Vert \gr \mathsf{P}_f^{\del} \Vert_{L^2(\Omega_T)} \leq C,
\end{equation}
with a constant $C$ which is independent of $\del$.

Now we choose $\varphi = \theta_f^\del - \theta_{\Gamma}$ in \eqref{H-17-weak}. This yields
\begin{equation}
\label{H-apr-1}
\begin{split}
\ds
\int_0^T \langle \pt_t (\Phi^\del S_f^\del +
[ (\mathcal{P}(S_f^\del) -  \mathcal{P}(S_{f}^0))*\omega^{\del} ]), \theta_f^\del - \theta_\Gamma\rangle dt \\ %[4mm]
\ds
+ \int_{\Omega_T} \mathbb{K}^{\star,\del} \lm_f(S_f^\del) \gr \theta_f^\del \cdot \gr \theta_f^\del \, dx dt
+ \int_{\Omega_T} \mathbb{K}^{\star,\del} \lm_{w,f}(S_f^\del)\gr \mathsf{P}_f^{\del} \cdot  \gr \theta_f^\del \, dx dt \\ %[4mm]
\ds
= \int_{\Omega_T} \mathbb{K}^{\star,\del} \lm_f(S_f^\del) \gr \theta_f^\del \cdot \gr \theta_\Gamma\, dx dt
+ \int_{\Omega_T} \mathbb{K}^{\star,\del} \lm_{w,f}(S_f^\del)\gr \mathsf{P}_f^{\del} \cdot \gr \theta_\Gamma\, dx dt. \\ %[4mm]
\end{split}
\end{equation}
The integral terms in the equality \eqref{H-apr-1} are denoted by $X_1, X_2, \ldots, X_5$, respectively.
Assume for the moment that the function $S_f^\del$ is sufficiently regular in time. Then we can write
 $X_1 = Y_1 + Y_2$. For $Y_1$, by \Ab{5}, we have:
\begin{equation}
\label{H-apr-2}
\begin{split}
\ds
Y_1 \eqdef & \int_0^T \int_\Omega  \pt_t (\Phi^\del S_f^\del) (\theta_f^\del - \theta_\Gamma)\, dx dt\\
%   = & \int_0^T \int_\Omega  \Phi^\del \pt_t S_f^\del \theta_f^\del\, dx dt -
%       \left(\int_\Omega   \Phi^\del S_f^\del(T) \theta_\Gamma(T)\, dx -
%       \int_\Omega \Phi^\del S_f^\del(0)\theta_\Gamma(0)\, dx\right)
%     + \int_0^T \int_\Omega   \Phi^\del S_f^\del \pt_t \theta_\Gamma\, dx dt\\
    = & \int_\Omega \Phi^\del ( H(\theta_f^\del(T)) -S_f^\del(T) \theta_\Gamma(T) )\, dx
     -  \int_\Omega \Phi^\del (H(\theta_f^\del(0)) -S_f^\del(0)\theta_\Gamma(0))\, dx \\
      + & \int_0^T \int_\Omega \Phi^\del S_f^\del \pt_t \theta_\Gamma\, dx dt
\geq  - \Phi^\del\, \left( 4\, \theta_f^*\, |\Omega| + \Vert \pt_t \theta_\Gamma \Vert_{L^1(\Omega_T)} \right),
\end{split}
\end{equation}
where the function $H$ is defined by
 \[ \ds  H(\theta) \eqdef \int_0^\theta {\cal B}_f'(r) r\, dr. \]
Moreover, it is easy to see,
\[
  \ds |H(\theta_f^\del)| = |{\cal B}_f(\theta_f^\del) \theta_f^\del - \int_0^{\theta_f^\del} {\cal B}_f(r)\, dr|
\leq 2\, \theta_f^*.
\]
For $Y_2$ part of the $X_1$ we obtain, using integration by parts,
\[
  Y_2 \eqdef  \int_0^T \int_\Omega
  \pt_t \left( (\mathcal{P}(S_f^\del) - \mathcal{P}(S_{f}^0))*\omega^{\del} \right) (\theta_f^\del - \theta_\Gamma) \, dx  dt
=   Y_2^1 + Y_2^2 + Y_2^3,
  \]
with
\begin{equation*}
\label{H-apr-3-01}
|Y_2^1| % \leq \int_\Omega \int_0^T \omega^{\del}(\tau) d\tau \, \theta_f^* dx
        \leq 2\, \del\, C_m^{max}, \quad
|Y_2^3| % \leq \int_0^T \int_\Omega \int_0^t \omega^{\del}(\tau) d\tau\, |\pt_t \theta_{\Gamma}(t)|\, dt\, dx
        % \leq 2\, \del\, C_m^{max}\, \int_0^T \int_\Omega 2\, \sqrt{T} |\pt_t \theta_{\Gamma}(t)|\, dt\, dx
        \leq 4\, \del\, \sqrt{T}\, C_m^{max}\, \Vert \pt_t \theta_{\Gamma} \Vert_{L^1(\Omega_T)}
\end{equation*}
and
\begin{equation*}
\label{H-apr-3-02}
Y_2^2 \eqdef - \int_0^T \int_\Omega \int_0^t \left( \mathcal{P}(S_f^\del(t-\tau)) - \mathcal{P}(S_{f}^0) \right)\,
         \omega^{\del}(\tau)\, d\tau\, \pt_t \theta_f^\del(t)\, dt\, dx.
\end{equation*}

By changing the order of the time integration and integrating by parts in the term $Y_2^2$, we can write
$\ds Y_2^2 = Y_2^{2,1} + Y_2^{2,2}$, where
\begin{equation*}
\label{H-apr-3-03}
|Y_2^{2,1}| \leq 8\, \del\, C_m^{max}\, \theta_f^*\, |\Omega|\, \sqrt{T}
\end{equation*}
and
\begin{equation*}
\label{H-apr-3-04}
Y_2^{2,2} \eqdef \int_0^T \int_\Omega \int_{\tau}^T \pt_t \mathcal{P}(S_f^\del(t-\tau)) \, \theta_f^\del(t)\ dt \,
         \omega^{\del}(\tau)\, dx\, d\tau.
\end{equation*}
Finally, the term $Y_2^{2,2}$ can be written as $\ds Y_2^{2,2} = A + B,$
where
\begin{equation*}
\label{H-apr-3-05}
|A| \leq 4\, \del\, C_m^{max}\, \theta_f^*\, |\Omega|\, \sqrt{T}
\end{equation*}
and
\begin{equation*}
\label{H-apr-3-06}
B \eqdef - \int_0^T \int_\Omega \pt_{\tau}
      \left( \int_{\tau}^T \mathcal{P}(S_f^\del(t-\tau)) \, \theta_f^\del(t)\ dt \right)\,
      \omega^{\del}(\tau)\, dx\, d\tau.
\end{equation*}
It can be proved, as in \cite{Arb-exist}, that for any $\tau \in [0,T]$ it holds
\begin{equation}
\label{h-del-1}
h^{\del}(\tau) \eqdef \int_{\tau}^T \mathcal{P}(S_f^\del(t-\tau)) \, \theta_f^\del(t)\ dt \leq h(0)
\end{equation}
and from \eqref{h-del-1} it follows that
\begin{equation}
\label{h-del-2}
\pt_{\tau} h(\tau) \leq 0 \textrm{ in } [0,T].
\end{equation}
Then  $B \geq 0$
which gives $Y_2^{2,2} \geq A \geq - 4\, \del\, C_m^{max}\, \theta_f^*\, |\Omega|\, \sqrt{T}.$
Summing all the obtained inequalities, we have for the first term in \eqref{H-apr-1} the estimate:
\begin{align*}
   X_1 \geq - C \left( \Phi^\del + \del \right),
\end{align*}
where constant $C$ depends on
$C_m^{max}$, $|\Omega|$, $T$, $\theta_f^*$, $\Vert \pt_t \theta_{\Gamma} \Vert_{L^1(\Omega_T)}$.
These calculations are applicable for regularized in time $S_f^\del$ but they remain true
for the desired $S_f^\del$ by a passage to the limit as the regularization parameter tends to $0$.

We treat the terms $X_2$,\ldots,$X_5$ in a standard way using the Cauchy-Schwartz inequality and
the already obtained estimate \eqref{H-discrete-5-1}.
Finally, we  have
\begin{equation}
\label{H-apr-9}
\ds
L_0\, \del\, \hat{k}_m^1\, \| \gr \theta_f^{\del}\|^2_{L^2(\Omega_T)}
\leq C\, \del + C\, \del\, \hat{k}_m^2 \, (1 + \| \gr \theta_f^{\del}\|_{L^2(\Omega_T)})
\end{equation}
and, therefore,
\begin{equation*}
\label{H-apr-10}
\| \gr \theta_f^{\del}\|_{L^2(\Omega_T)} \leq C,
\end{equation*}
with a constant $C$ which is independent of $\del$.
The estimate \eqref{unif-est-GP-2} follows in the standard way from \eqref{unif-est-GP-1}.
This completes the proof of  Proposition~\ref{Prop-uniform-est}.

\end{proof}

\begin{lemma}
\label{lemma-mod_cont-cont}
There exists a constant $C$ which is independent of $\del$ and $h$ such that, as $h\to0$,
\begin{equation}
\label{comp-5}
\int_{h}^T\,\int_{\Omega}\,
\big(S_f^\del(x,t) - S_f^\del(x,t-h)\big)\big( \theta_f^\del(x,t) - \theta_f^\del(x,t-h)  \big)\,dx\,dt \leq \, C\,\sqrt{h}.
\end{equation}
\end{lemma}
\begin{proof}
Let us first note that for integrable functions $G_1, G_2$ and for $0 < h < T/2$ it holds
\begin{equation}
\label{comp-6}
  \int_0^T G_1(t) \int_{\max(t,h)}^{\min(t+h,T)} G_2(\tau)d\tau dt = \int_h^T G_2(t) \int_{t-h}^t G_1(\tau)d\tau dt.
\end{equation}
We define the test function in  \eqref{H-17-weak} by
\begin{equation*}
\ds
\varphi = \varphi^{\del,h}(x,t)
= \int_{\max(t,h)}^{\min(t+h,T)} \left( ( \theta_f^\del(x,\tau) - \theta_\Gamma(x,\tau) ) -
(\theta_f^\del(x,\tau-h) - \theta_\Gamma(x,\tau-h)) \right) \, d\tau.
\end{equation*}
Then $\varphi \in L^2(0,T;V)$. Plugging it in \eqref{H-17-weak} we have:
\begin{equation}
\label{comp-8}
\begin{split}
\int\limits_{0}^T\langle \pt_t(\Phi^\del S_f^\del + [ \mathcal{P}(S_f^\del) - & \mathcal{P}(S_f^0) ] *  \omega^\del ), \varphi^{\del,h} \rangle dt = \\
- &  \int_{\Omega_T}  \mathbb{K}^{*,\del} ( \lm_f(S_f^\del) \gr \theta_f^\del
  + \lm_{w,f}(S_f^\del) \gr \mathsf{P}_f^{\del} )\cdot \gr \varphi^{\del,h}\, dx dt.
\end{split}
\end{equation}
By using \eqref{comp-6} the left-hand side term can be written as
\begin{equation}
\label{comp-9}
\begin{split}
\int\limits_{0}^T\langle \partial_t(\Phi^\del S_f^\del + &
 [ \mathcal{P}(S_f^\del) -  \mathcal{P}(S_f^0) ] * \omega^\del ), \varphi^{\del,h}\rangle dt \\
\ds = & \int_h^T\int_\Omega \Phi^\del (S_f^\del(x,t)- S_f^\del(x,t-h))(\theta_f^\del(x,t)-\theta_f^\del(x,t-h))\, dx dt \\
\ds - & \int_h^T\int_\Omega \Phi^\del (S_f^\del(x,t)- S_f^\del(x,t-h))(\theta_\Gamma (x,t)- \theta_\Gamma(x,t-h))\, dx dt \\
\ds + & \int_h^T\int_\Omega [(\theta_f^\del(x,t) - \theta_f^\del(x,t-h))- (\theta_\Gamma  (x,t) - \theta_\Gamma  (x,t-h))]X^\del_h(x,t)\, dx dt,
\end{split}
\end{equation}
where
\begin{equation}
\label{comp-10}
\begin{split}
   X^\del_h (x,t)
  &= ([ \mathcal{P}(S_f^\del) -  \mathcal{P}(S_f^0) ] * \omega^\del) (x,t) -
 ([ \mathcal{P}(S_f^\del) -  \mathcal{P}(S_f^0) ] * \omega^\del) (x,t-h)\\
 &=  \int_h^t [  \mathcal{P}(S_f^\del(x,t-\tau) -  \mathcal{P}(S_f^0(x)) ][\omega^\del(\tau) -\omega^\del(\tau -h)]\, d\tau \\
 & +\int_0^h [  \mathcal{P}(S_f^\del(x,t-\tau) -  \mathcal{P}(S_f^0(x)) ]\omega^\del(\tau)\, d\tau.
\end{split}
\end{equation}
Let us denote the integral terms at the right-hand side of the equality \eqref{comp-9} by $Z_1, Z_2, Z_3$,
respectively.
For $Z_2$ we have by using \Ab{5}:
\begin{equation}
\label{comp-11}
\begin{split}
  | Z_2|\, \leq & \,  \Phi^\del\  \int_h^T\int_\Omega|\theta_\Gamma  (x,t) - \theta_\Gamma  (x,t-h)|\, dx dt
  \leq  h\, \Phi^\del\,  \Vert \pt_t \theta_\Gamma\Vert_{L^1(\Omega_T)}
   \leq  C \, h\, \del\,  \Vert \pt_t \theta_\Gamma\Vert_{L^1(\Omega_T)},
\end{split}
\end{equation}
since $\Phi^\del/\del \leq C$, uniformly with respect to $\del$.
Next, due to $0\leq \mathcal{P} \leq 1$ and $\omega^\del > 0$ we have
\begin{align*}
  | X^\del_h(x,t) | \leq  2 \int_0^h \omega^\del(\tau)\, d\tau = 4 \, \del \, C_m^{max}\, \sqrt{h}
\end{align*}
and therefore
\begin{equation}
\label{comp-12}
  | Z_3 | \leq  8 \vert \Omega_T \vert \, \theta_f^* \,  C_m^{max}\, \del\, \sqrt{h} .
\end{equation}
Finally, we apply \eqref{comp-6} with $$G_1(t)\equiv1,\
G_2(\tau)=| (\gr \theta_f^\del(x,\tau) - \gr \theta_\Gamma(x,\tau) ) -
( \gr \theta_f^\del(x,\tau-h) - \gr \theta_\Gamma(x,\tau-h) ) |^2$$
to establish
\begin{equation}
\label{comp-13}
\ds
  \Vert \gr \varphi^{\del,h}\Vert_{L^2(\Omega_T)} \leq
  2\, h \Vert \gr (\theta_f^\del - \theta_\Gamma) \Vert_{L^2(\Omega_T)}\leq C\,h,
\end{equation}
where we have used the uniform a priori estimate \eqref{unif-est-GP-1}.
From estimate \eqref{comp-13} and the uniform bounds \eqref{unif-est-GP-1} we hence obtain
\begin{equation}
\label{comp-14}
\begin{split}
  \left | \int_{\Omega_T} \mathbb{K}^{*,\del} ( \lm_f(S_f^\del) \gr \theta_f^\del + \lm_{w,f}(S_f^\del) \gr \mathsf{P}_f^{\del} )\cdot
 \gr \varphi^{\del,h}\, dx dt \right|
 \leq C\, h\, \del\, \hat{k}_m^2.
\end{split}
\end{equation}
Collecting the estimates \eqref{comp-11}, \eqref{comp-12}, \eqref{comp-14}, % and \eqref{comp-9},
from \eqref{comp-8} we get
\begin{align*}
\int_h^T\int_\Omega \Phi^\del (S_f^\del(x,t)&- S_f^\del(x,t-h))(\theta_f^\del(x,t) - \theta_f^\del(x,t-h))\, dx dt
 \leq   C\, \del\, \sqrt{h}.
 \end{align*}
Now, the desired estimate follows from $\Phi^\del \geq c \del$, for some $c$ independent of $\del$,
and the monotonicity of $S\mapsto \beta_f(S)$.

\end{proof}

%%%%%%%%%%%%%%%%%%%%%%%%%%%%%%%%%%%%%%%%%%%%%%%%%%%%%%%%%%%%%%%%%
\subsection{The fully homogenized model}
%{The simplified fully homogenized model for immiscible incompressible two-phase flow 
%in double porosity media with thin fissures}
\label{sect:full_homogen}
%%%%%%%%%%%%%%%%%%%%%%%%%%%%%%%%%%%%%%%%%%%%%%%%%%%%%%%%%%%%%%%%%

In this subsection we present the fully homogenized model for immiscible incompressible
two-phase flow in double porosity media with thin fractures.
First we state the convergence results holding as $\del\to0$.

\begin{theorem}
\label{tm-compact}
Let assumptions \Ab{1}--\Ab{5} be fulfilled.
Let $(\mathsf{P}_f^{\del}, \theta_{f}^{\del})$
be a weak solution of the problem
\eqref{H-delta-lin-GP}-\eqref{H-15}
and let $S_f^{\del} = \mathcal{B}_f(\theta_f^{\del})$.
Then there exist functions $\mathsf{P}_f\in L^2(0,T;V) + P_{\Gamma}$ and $\theta_f\in L^2(0,T;V) + \theta_{\Gamma}$
such that, up to a subsequence, it holds
\begin{align}
\label{comp-1}
\mathsf{P}_f^{\del}(x, t) & \tow \mathsf{P}_f (x, t) \quad {\rm \ weakly\ in\ }\,\,  L^2(0, T;  H^1(\Omega)),\\
\label{comp-2}
\theta_f^{\del}(x, t) & \tow \theta_f (x, t) \quad {\rm \ weakly\ in\ }\,\,  L^2(0, T;  H^1(\Omega))
\end{align}
as $\del\to0$. Moreover,  $0 \leq \theta_f(x,t) \leq \theta_f^{\star}$ a.e. in $\Omega_T.$
Furthermore,
\begin{equation}
\label{comp-3-2}
S_f^{\del}(x,t) \to S_f(x,t) \quad {\rm \ strongly\ in\ }\,\, L^2(\Omega_T) {\rm \ and\ a.e.\ in\ }\,\, \Omega_T,
\end{equation}
where $S_f = \mathcal{B}_f(\theta_f).$
Here $(\mathsf{P}_f, \theta_f)$ is a weak solution in $\Omega_T$ of problem:
\begin{equation}
\label{FH-1}
\left\{
\begin{array}[c]{ll}
\ds
{\rm div}\, \left( \lm_{f} (S_f){k}^{\star} \gr P_{f} \right) = 0, \\[2mm]
\ds
\Phi_f \frac{\pt S_f}{\pt t}
- {\rm div}\, \bigg( {k}^{\star} \left( \lm_{f} (S_f) \gr \theta_{f}
+ \lm_{w,f} (S_f) \gr P_{f} \right) \bigg) \\[2mm]
\ds \mbox{}\qquad\qquad \qquad =
- \frac{C_m(x)}{d}\,\frac{\pt}{\pt t}\, \big[\,\big(\mathcal{P}(S_f) - \mathcal{P}(S_f^{0}) \big)
\ast \frac{1}{\sqrt{t}} \big], \\
S_f = \mathcal{B}_f(\theta_f). % \quad {\rm in}\,\, \Omega_{T}
\end{array}
\right.
\end{equation}
The boundary conditions for system \eqref{FH-1} are given by:
\begin{equation}
\label{FH-4}
\left\{
\begin{array}[c]{ll}
P_{f} = P_{\Gamma} \quad {\rm and} \quad \theta_{f} = \theta_{\Gamma} \quad {\rm on} \,\,
\Gamma_{inj} \times (0, T), \\[2mm]
 k^{\star} \lm_{f}(S_f) \gr P_{f} \cdot \bs{\nu} = 0 \quad {\rm on} \,\,\Gamma_{imp} \times (0, T),
\\[2mm]
 k^{\star} \left( \lm_{f} (S_f) \gr \theta_{f}
+ \lm_{w,f} (S_f) \gr P_{f} \right) \cdot \bs{\nu} = 0 \quad {\rm on} \,\,
\Gamma_{imp} \times (0, T),
\end{array}
\right.
\end{equation}
and the initial condition is
\begin{equation}
\label{FH-5}
\theta_f(x, 0) = \theta^{0}_f(x) \,\, {\rm in} \,\, \Omega.
\end{equation}
The effective permeability tensor $k^{\star}$ and the function $C_m$ are given by \eqref{H-2-unif-5-1}
and \eqref{w_Ddelta-def}.
\end{theorem}
\begin{proof}
Weak convergences in \eqref{comp-1} and \eqref{comp-2} follow from \eqref{unif-est-GP-1}.
The boundedness of $\theta_f^{\del}$ and $S_f = \mathcal{B}_f(\theta_f)$ follows directly
from the strong convergence \eqref{comp-3-2}.
In order to prove \eqref{comp-3-2} we use Lemma~\ref{lemma-mod_cont-cont}
and Lemma~1.9 from \cite{AL}, which we repeat for reader's convenience:
\begin{lemma}
\label{lemma-AL-modif}
Suppose that the sequence $(u_{\del})_{\del}$ converges weakly to $u$ in $L^2(0,T;H^1(\Omega))$.
Let $F$ be a continuous, monotone and bounded function in $\R$.
Assume that
\begin{equation}
\label{comp-16}
\int_{h}^T\,\int_{\Omega}\,
\big(F(u_\del(x,t)) - F(u_\del(x,t-h)\big) \big( u_\del(x,t) - u_\del(x,t-h)  \big)\, dx\, dt \leq \, C\, \varpi(h),
\end{equation}
for some continuous function $\varpi$ such that $\varpi(0)=0$,
and with a constant $C$  independent of $h$ and $\del$.
Then $F(u_\del)$ converges to $F(u)$ strongly in $L^2(\Omega_T)$.
\end{lemma}

Now we apply Lemma~\ref{lemma-AL-modif} to the sequence $(\theta_f^\del)_{\del}$
in the role of $(u_\del)_{\del}$.
The conditions on the function $F(z)=\mathcal{B}_f(z)$ in Lemma~\ref{lemma-AL-modif}
hold from the definition of $\mathcal{B}_f$.

Due to \eqref{unif-est-GP-2}, up to a subsequence we have
\begin{equation}
\label{comp-conv_h-1}
\frac{1}{\del} \pt_t\left(\Phi^\del S_f^\del + \big( \mathcal{P}(S_f^\del) - \mathcal{P}(S_f^0) \big) * \omega^\del \right)
\tow \Psi \quad\text{\rm \ weakly in\ }\,\, L^2(0,T;V')
 \end{equation}
for some $\Psi \in L^2(0,T;V')$.
The strong convergence of $S_f^{\del}$ in \eqref{comp-3-2} allows to identify the limit as
\[ \Psi = \pt_t\left(d\Phi_f S_f + C_m(x)\, [ \mathcal{P}(S_f) - \mathcal{P}(S_f^0) ] * \frac{1}{\sqrt{t}} \right) \in L^2(0,T;V'). \]
We can now pass to the limit as $\del\to 0$ in the equations \eqref{H-16-weak}, \eqref{H-17-weak} and \eqref{H-weak-init}, after
division by $d\,\del$, and obtain  for any $\zeta, \varphi \in L^2(0,T;V)$:
\begin{equation}
\label{H-16-weak-final}
\int\limits_{\Omega_T} \lm_{f} (S_f){k}^{\star} \gr \mathsf{P}_f \cdot \gr\zeta\, dx\, dt = 0,
\end{equation}
\begin{equation}
\label{H-17-weak-final}
\begin{split}
  \int\limits_{0}^T \langle  \pt_t \Big( \Phi_f S_f + \frac{C_m(x)}{d}
  [ \mathcal{P}(S_f) & - \mathcal{P}(S_f^0) ] * \frac{1}{\sqrt{t}} \Big) , \varphi \rangle  dt\\
&+
\int\limits_{\Omega_T} {k}^{\star} \left( \lm_{f} (S_f) \gr \theta_{f}
+ \lm_{w,f} (S_f) \gr \mathsf{P}_f \right) \cdot\gr\varphi\, dx\, dt = 0.
\end{split}
\end{equation}
In the initial condition  we take the test function $ \varphi \in L^2(0,T;V) \cap W^{1,1}(0,T;L^1(\Omega))$ such that $\varphi(\cdot, T) = 0$ in $\Omega$,
and obtain
\begin{equation}
\label{H-weak-init-final}
\begin{split}
  \int\limits_{0}^T \langle  \pt_t \Big( \Phi_f S_f +  & \frac{C_m(x)}{d} [ \mathcal{P}(S_f) - \mathcal{P}(S_f^0) ] * \frac{1}{\sqrt{t}} \Big) , \varphi \rangle  dt\\
+ & \int\limits_{\Omega_T}
\left( \Phi_f ( S_f - S_f^0 ) + \frac{C_m(x)}{d} \big(\mathcal{P}(S_f) - \mathcal{P}(S_f^{0}) \big) \ast \frac{1}{\sqrt{t}} \right) \pt_t \varphi\, dx\, dt = 0.
\end{split}
\end{equation}

Finally, we see that obtained equations \eqref{H-16-weak-final},
\eqref{H-17-weak-final} and \eqref{H-weak-init-final} represent
a weak formulation of the problem \eqref{FH-1}-\eqref{FH-5}.
This completes the proof of Theorem~\ref{tm-compact}.

\end{proof}

Finally, we can transform the system \eqref{FH-1} into the phase formulation
by reintroducing the phase pressures as follows:
\begin{equation*}
\left\{
\begin{array}[c]{ll}
\ds
\Phi_f \frac{\pt S_f}{\pt t}
-{\rm div}\, \left( {k}^{\star} \lm_{w,f} (S_f) \gr P_{w,f} \right) =
- \frac{C_m(x)}{d}\,\frac{\pt}{\pt t}\, \big[\,\big(\mathcal{P}(S_f) - \mathcal{P}(S_f^{0}) \big)
\ast \frac{1}{\sqrt{t}} \big] \,\,  {\rm in}\,\, \Omega_{T}, \\[6mm]
\ds
-\Phi_f \frac{\pt S_f}{\pt t}
-{\rm div}\, \left( {k}^{\star} \lm_{n,f} (S_f) \gr P_{n,f} \right) =
\frac{C_m(x)}{d}\,\frac{\pt}{\pt t}\, \big[\,\big(\mathcal{P}(S_f) - \mathcal{P}(S_f^{0}) \big)
\ast \frac{1}{\sqrt{t}} \big] \,\,  {\rm in}\,\, \Omega_{T}.
\end{array}
\right.
\end{equation*}
The boundary conditions for this system are given by:
\begin{equation*}
\left\{
\begin{array}[c]{ll}
P_{w,f} = P_{w,\Gamma} \quad {\rm and} \quad P_{n,f} = P_{n,\Gamma} \quad {\rm on} \,\,
\Gamma_{inj} \times (0, T), \\[3mm]
- k^{\star} \lm_{w,f}(S_f) \gr P_{w,f} \cdot \bs{\nu} =
- k^{\star} \lm_{n,f}(S_f) \gr P_{n,f} \cdot \bs{\nu} = 0 \quad {\rm on} \,\,
\Gamma_{imp} \times (0, T),\\
\end{array}
\right.
\end{equation*}
and the initial condition is
\begin{equation*}
S_f(x, 0) = S^{0}_f(x) \,\, {\rm in} \,\, \Omega.
\end{equation*}
We note that this model is fully homogenized, and the effective coefficients of the model
are given by simple relations (see \eqref{H-2-unif-5-1} and \eqref{w_Ddelta-def}):
$k^\star = (d-1)k_f/d$ and $C_m(x) = {6 \sqrt{\Phi_m k_m \psi_m(x)}}/{\sqrt{\pi}}$.

\section*{Acknowledgments}
This work was partially supported by University of Zagreb, grant 202600.
Most of the work on this paper was done when
Leonid Pankratov was visiting Department of Mathematics, Faculty of Science, University of Zagreb.
We thank Faculty of Science for hospitality.

%%%%%%%%%%%%%%%%%%%%%%%%%%%%%%%%%%%%%%%%%%%%%%%%%%%%%%%%%%%%%%%%%%%%%%%%%%%%%%%%%%

%%%%%%%%%%%%%%%%%%%%%%%%%%%%%%%%%%%%%%%%%%%%%%%%%%%%%%%%%%%%%%%%%%%%%%%%%%%%%%%%%%

\end{document}